\documentclass{amsart}

\usepackage{graphicx,xcolor}

\numberwithin{equation}{section}
\usepackage{cite}
\usepackage{amsmath,amssymb,mathtools}
\usepackage{aliascnt}
\usepackage[hypertexnames=false]{hyperref}
\usepackage{enumitem}
\usepackage{placeins}
\usepackage[initials,nobysame]{amsrefs}

\hypersetup{
    colorlinks=true,
    allcolors=red!70!black,
}
\setlist[enumerate,1]{label=\upshape{(\roman*)},ref=\roman*}
\setlist[enumerate,2]{label=\upshape{(\alph*)},ref=\alph*}

\newtheorem{theorem}{Theorem}[section]
\newaliascnt{proposition}{theorem}
\newtheorem{proposition}[proposition]{Proposition}
\aliascntresetthe{proposition}
\newaliascnt{corollary}{theorem}
\newtheorem{corollary}[corollary]{Corollary}
\aliascntresetthe{corollary}
\newaliascnt{lemma}{theorem}
\newtheorem{lemma}[lemma]{Lemma}
\aliascntresetthe{lemma}
\theoremstyle{definition}
\newaliascnt{definition}{theorem}
\newtheorem{definition}[definition]{Definition}
\aliascntresetthe{definition}
\newtheorem*{acknowledgements}{Acknowledgements}
\newtheorem*{AI}{AI Usage Disclosure}
\theoremstyle{remark}
\newaliascnt{remark}{theorem}

\aliascntresetthe{remark}

\usepackage[capitalize]{cleveref}
\usepackage{autonum}
\crefname{theorem}{theorem}{theorems}
\Crefname{theorem}{Theorem}{Theorems}
\crefname{proposition}{proposition}{propositions}
\Crefname{proposition}{Proposition}{Propositions}
\crefname{corollary}{corollary}{corollaries}
\Crefname{corollary}{Corollary}{Corollaries}
\crefname{lemma}{lemma}{lemmas}
\Crefname{lemma}{Lemma}{Lemmas}
\crefname{definition}{definition}{definitions}
\Crefname{definition}{Definition}{Definitions}
\crefname{remark}{remark}{remarks}
\Crefname{remark}{Remark}{Remarks}

\newcommand{\R}{\mathbf{R}}
\newcommand{\Z}{\mathbf{Z}}

\newcommand{\T}{\mathbf{T}}

\newcommand{\Sph}{\mathbf{S}^2}
\newcommand{\E}{\mathcal{E}}
\newcommand{\B}{\mathcal{B}}
\renewcommand{\L}{\mathcal{L}}

\newcommand{\Rop}{\mathcal{R}}
\newcommand{\CK}{\mathcal{C}(K)}

\DeclareMathOperator{\TC}{TC}
\DeclareMathOperator{\bri}{bri}
\DeclareMathOperator{\bra}{bra}

\title[Milnor's inequality and circular elastic knots]{Milnor's inequality and circular elastic knots}
\author[T.~Miura]{Tatsuya Miura}
\address[T.~Miura]{Department of Mathematics, Graduate School of Science, Kyoto University, Kitashirakawa Oiwake-cho, Sakyo-ku, Kyoto 606-8502, Japan}
\email{tatsuya.miura@math.kyoto-u.ac.jp}
\date{September 23, 2026}

\begin{document}
\begin{abstract}
We establish the extension of Milnor's inequality, relating total curvature with the bridge index, to the $C^1$-closure of a knot class, without any restriction on the self-intersections of the limit curve. Together with recent work of Reiter--von der Mosel, this resolves the circular elastic knot conjecture, first predicted by Gallotti--Pierre-Louis in 2007 and then formulated as a mathematical conjecture by Gerlach--Reiter--von der Mosel in 2017. More precisely, if the bridge and braid indices of a tame knot class coincide, then the multiply covered circle is the unique elastic knot.
\end{abstract}

\maketitle

\section{Introduction}\label{sec:introduction}

The variational theory of elastic curves goes back to D.~Bernoulli and Euler. In its classical form, an elastica is a critical point of the bending energy among immersed curves of fixed length. The theory of closed elasticae in Euclidean space was developed systematically by Langer and Singer in the 1980s; in particular, they proved that the only stable closed elastica in $\R^3$ is the once covered circle \cite{LS85}. 

Physical closed elastic wires, however, can exhibit many stable configurations other than the once covered circle. To describe such shapes, one must take into account global constraints such as self-avoidance and knot type, leading naturally to the theory of elastic knots. We refer to \cite{MR5054189} for a brief survey of classical elastica theory and related questions involving self-intersections.

Here we consider the mathematical model of elastic knots introduced by Gerlach--Reiter--von der Mosel \cite{GRM}, in which the bending energy is penalized by a small multiple of the ropelength.
Let $\gamma\in H^2(\T;\R^3)$ be a regular closed curve of $H^2$-Sobolev class, where $\T:=\R/\Z$.
The bending energy is defined by
\begin{equation}\label{eq:intro-bending}
    \B[\gamma]:=\int_\gamma|\kappa|^2\,ds,
\end{equation}
where $s$ denotes arclength and $\kappa:=\partial_s^2\gamma$ the curvature vector.
For the self-avoidance term, let $\Rop[\gamma]$ denote the ropelength of $\gamma$, see \eqref{eq:intro-ropelength} below.
For $\vartheta>0$, we define
\begin{equation}\label{eq:intro-penalized}
    \E_\vartheta[\gamma]:=\B[\gamma]+\vartheta\Rop[\gamma].
\end{equation}

Throughout the paper, $K$ denotes a tame knot class. 
Following \cite{GRM} (but without fixing a basepoint), we use the normalized class
\begin{equation}\label{eq:intro-class}
    \CK:=\left\{\gamma\in H^2(\T;\R^3):
    \text{$\gamma$ is an embedding representing $K$, with $|\gamma'|=1$} \right\}.
\end{equation}
In particular, every curve in $\CK$ has length one. 

Gerlach--Reiter--von der Mosel proved that $\E_\vartheta$ admits a global minimizer in $\CK$ for every $\vartheta>0$, and that any sequence of such minimizers $\gamma_{\vartheta_j}$ with $\vartheta_j\to0$ admits a $C^1$-convergent subsequence after translations \cite[Theorems~2.1 and~2.2]{GRM}.
These facts motivate the following definition of (energy minimal) elastic knots.

\begin{definition}[Elastic knot]
    A closed curve $\gamma:\T\to\R^3$ is called an \emph{elastic knot} for a tame knot class $K$ if $\gamma$ is obtained as a $C^1$ limit of global minimizers $\gamma_{\vartheta_j}$ of $\E_{\vartheta_j}$ in $\CK$ along some sequence $\vartheta_j\to0$.
\end{definition}

Note that every elastic knot belongs to $H^2(\T;\R^3)$ and is unit-speed, by the uniform bending-energy bound for the defining sequence of minimizers and its $C^1$ convergence.

A central problem in elastic knot theory is to determine the elastic knot associated with a given knot class and, more broadly, to understand how topological invariants of knots influence the shapes of elastic knots.
Recall that the bridge index $\bri(K)$ measures the least possible number of bridges of a representative of $K$, while the braid index $\bra(K)$ is the minimum number of strands among all closed braid representatives of $K$, and in general $\bri(K)\leq\bra(K)$; see, e.g., \cite[Chapters~10 and~16]{BZH} for standard background.
The circular elastic knot conjecture, first predicted by Gallotti--Pierre-Louis \cite{GPL} and then formulated by Gerlach--Reiter--von der Mosel \cite[Conjecture~7.1]{GRM}, asserts that, whenever $\bri(K)=\bra(K)$, the corresponding multiply covered circle is the unique elastic knot.

In this paper we resolve this conjecture.

\begin{theorem}[Circular elastic knots]\label{thm:main}
    Let $K$ be a tame knot class with $\bri(K)=\bra(K)=a$. Then, up to Euclidean isometries, the unique elastic knot for $K$ is given by the $a$-fold covered circle.
\end{theorem}

This result determines elastic knot configurations for a large family of knot classes.
Knot classes satisfying $\bri(K)=\bra(K)$ are called BB knot classes and systematically studied in \cite{DER}.
In particular, the number of (one-component) BB knots grows exponentially with respect to the crossing number \cite[Theorem~5.1]{DER}.

The proof of \Cref{thm:main} naturally consists of two main steps: first, extending Milnor's inequality to the full $C^1$-closure of a knot class, and second, identifying the limiting elastic knot with the multiply covered circle. After the present manuscript had been essentially completed, the author became aware, during the final proofreading stage and prior to public release, of Reiter--von der Mosel's recent work \cite{ReiterVdM2026}, which provides an argument for the second step. The author had independently obtained a proof of this step by a different approach, based on a quantitative one-sided estimate for ropelength under smooth variations. Since the two arguments nevertheless overlap substantially, we have omitted our original treatment from the present paper and instead use the corresponding result of \cite{ReiterVdM2026}. The alternative argument may be presented elsewhere.

The main contribution of this paper is thus the $C^1$-extension of Milnor's inequality, see \Cref{thm:milnor} below.
Once this is established, the result in \cite{ReiterVdM2026} directly implies \Cref{thm:main}.
In what follows, we provide further background on the problem and then prove Milnor's inequality needed for this conclusion.

\subsection{Background}

Gallotti--Pierre-Louis \cite{GPL} combined theoretical arguments with experimental and numerical evidence for the limiting shapes of elastic knots.
In particular, they observed that Milnor's inequality gives
\[
\inf_{\gamma\in\CK}\B[\gamma]\geq (2\pi\bri(K))^2,
\]
while the braid representation easily yields
\[
\inf_{\gamma\in\CK}\B[\gamma]\leq (2\pi\bra(K))^2.
\]
When $\bri(K)=\bra(K)$, these bounds match, leading them to predict a circular limiting configuration.

However, the above energy bounds do not directly ensure that the limiting configuration $\gamma$ of a minimizing sequence still satisfies $\B[\gamma]\geq (2\pi\bri(K))^2$, since the available compactness yields only $H^2$-weak convergence and hence lower semicontinuity of the bending energy, which is in the opposite direction.

Gerlach--Reiter--von der Mosel \cite{GRM} addressed this subtle issue through the ropelength-penalized formulation introduced above.
For an embedded curve $\gamma:\T\to\R^3$, the ropelength is defined as the scale-invariant quantity
\begin{equation}\label{eq:intro-ropelength}
    \Rop[\gamma]:=\frac{\L[\gamma]}{\Delta[\gamma]},
\end{equation}
where $\L[\gamma]:=\int_\gamma ds$ denotes the length and $\Delta[\gamma]$ denotes the thickness. 
Following Gonzalez--Maddocks \cite{GM99}, one definition of the thickness is given by
\begin{equation}\label{eq:intro-thickness}
    \Delta[\gamma]:=
    \inf_{\substack{x,y,z\in\gamma(\T)\\x\neq y\neq z \neq x}}
    \operatorname{rad}(x,y,z), \qquad \operatorname{rad}(x,y,z):=\frac{|x-y|\,|y-z|\,|z-x|}{2|(y-x)\times(z-x)|},
\end{equation}
where $\operatorname{rad}(x,y,z)$ denotes the circumradius of the triple $(x,y,z)$, understood as $\operatorname{rad}(x,y,z)=\infty$ when the three points are collinear. We set $\Rop[\gamma]:=\infty$ if $\Delta[\gamma]=0$.
Then they established the circular elastic knot conjecture for the $(2,q)$-torus knots, where $q$ is odd and $|q|\ge3$, in particular for the trefoil \cite[Theorem~1.1]{GRM}; in this case, the resulting elastic knot is the doubly covered circle.
Their proof first identifies the possible limits as tangential pairs of circles, and then excludes every pair other than the doubly covered circle by a quantitative crookedness estimate.
A key step is a $C^1$-extension of the F\'ary--Milnor theorem \cite[Theorem~A.1]{GRM}.

More recently, Reiter--von der Mosel \cite{ReiterVdM2026} proved the circular elastic knot conjecture, \emph{assuming} that Milnor's total-curvature inequality extends to the $C^1$-closure of a knot class. Our main result, \Cref{thm:milnor} below, provides precisely this missing input; the remaining variational and classification arguments are contained in \cite{ReiterVdM2026}.

A subtlety in elastic knot theory is how to capture the effect of thickness as it tends to zero.
We note that the vanishing-thickness constrained problem in Gallotti--Pierre-Louis \cite{GPL} and the vanishing-ropelength-penalty problem of Gerlach--Reiter--von der Mosel \cite{GRM} are distinct variational formulations; the motivation above concerns the latter formulation.
Earlier work of von der Mosel \cite{vdM} considered yet another formulation, involving a vanishing Coulomb-type self-repulsive penalty.
It is worth comparing those formulations in several contexts, particularly in view of the circular elastic knot problem.

Beyond the circular setting, Bartels--Reiter \cite{BR} numerically observed a planar candidate for the elastic knot in the figure-eight knot class $4_1$. In \cite{MR4861585}, an explicit critical teardrop--heart configuration matching this candidate was identified analytically, and was conjectured to give the elastic knot of class $4_1$.
The spherical elastic knot conjecture posed in \cite[Conjecture~7.2]{GRM} also remains open.
Further developments include the study of symmetric elastic knots by Gilsbach--Reiter--von der Mosel \cite{GRvM}.
A non-minimal but stable elastic configuration in the trivial knot class, termed the elastic propeller, is predicted in \cite{MR4631455} (see also \cite{MR5054189}).

We also mention that the equality $\bri(K)=\bra(K)$ is itself of independent knot-theoretic interest. 
Recently, Krishna--Morton conjectured that it holds for all $L$-space knots and all positive braid knots \cite[Conjectures~1.8 and~1.10]{KrishnaMorton}. 
They proved the equality for twist positive $L$-space knots \cite[Theorem~1.3]{KrishnaMorton}, and Himeno subsequently extended this result to all twist positive knots \cite[Theorem~1.2]{Himeno}.

\subsection{Extension of Milnor's inequality}

Now we discuss our main result that extends Milnor's inequality to the $C^1$-closure of a knot class. 
Recall that the classical total curvature for a smooth regular closed curve $\gamma$ is given by
\begin{equation}\label{eq:def-classical-total-curvature}
    \TC[\gamma]=\int_\T |\kappa| \,ds,
\end{equation}
and Milnor's inequality asserts that, if an embedding $\gamma$ represents $K$, then
\begin{equation}\label{eq:intro-total-bound}
    \TC[\gamma]\ge2\pi \bri(K).
\end{equation}

Here we extend Milnor's inequality to possibly non-embedded curves, which arise as $C^1$-limits of embedded curves representing $K$.
It is well known that the definition of total curvature extends to arbitrary irregular curves.
For any polygon $P$ in $\R^3$ (i.e., piecewise affine closed curve), we define
\begin{equation}
    \TC[P]:=(\text{the total variation of the tangent indicatrix of $P$ in $\Sph$}),
\end{equation}
and then for any curve $\gamma:\T\to\R^3$, 
\begin{equation}
    \TC[\gamma]:=\sup\{ \TC[P] : \text{$P$ is a polygon inscribed in the closed curve $\gamma$} \};
\end{equation}
see, e.g., \cite{Sullivan} for more details.

Our main result then reads as follows.

\begin{theorem}[$C^1$-extension of Milnor's inequality]\label{thm:milnor}
    Let $K$ be a tame knot class and $\gamma:\T\to\R^3$ be a regular $C^1$ closed curve with $\TC[\gamma]<\infty$.
    Suppose that there is a sequence of $C^1$ embeddings $\gamma_j\colon\T\to\R^3$ representing the knot class $K$ such that $\gamma_j\to\gamma$ in $C^1$. 
    Then
    \[
    \TC[\gamma]\ge2\pi \bri(K).
    \]
\end{theorem}

Gerlach--Reiter--von der Mosel \cite[Theorem~A.1]{GRM} proved the F\'ary--Milnor theorem $\TC[\gamma]\geq4\pi$ for the $C^1$-closure of any nontrivial knot class $K$, using Denne’s result on the existence of alternating quadrisecants.
Wacker \cite[Theorem~3.1]{Wacker} established \Cref{thm:milnor} assuming that the limit curve $\gamma$ has only finitely many isolated self-intersections, and posed the extension to arbitrary self-intersections as an open problem.
\Cref{thm:milnor} resolves this problem.

Wacker's proof of an extension of Milnor's inequality constructs a global recovery sequence whose total curvature converges to that of the given limit curve $\gamma$. 
Our argument takes a different strategy: we do \emph{not} prove the convergence of the total curvature, but instead, working direction by direction, we only recover the number of extrema of a generic height function. 
This weaker recovery is sufficient to extend Milnor's inequality and, crucially, avoids any finiteness or isolation assumption on the self-intersections.

Now we explain our proof ideas in detail.
Following Milnor \cite{Milnor}, for a regular $C^1$ closed curve $\gamma$ and a direction $v\in\Sph$, we define the crookedness of $\gamma$ in the direction $v$ by
\begin{equation}\label{eq:intro-crookedness}
    \mu[\gamma,v]
    :=\#\left\{t\in\T:
    t\text{ is a local maximum of }\langle\gamma(\cdot),v\rangle\right\},
\end{equation}
which may be infinite.
With this notation, the bridge index can be expressed as
\begin{equation}\label{eq:intro-bridge}
    \bri(K)=\min\left\{\mu[\eta,v]:
    \begin{array}{l}
        \eta\in C^\infty(\T;\R^3)\text{ an embedding representing }K,\\
        v\in\Sph
    \end{array}\right\},
\end{equation}
see \cite[Section~27.2]{MR4439733}.
In addition, for a curve of finite total curvature $\gamma$, Milnor's integral formula \cite[Theorem~3.1]{Milnor} relates $\TC[\gamma]$ to the crookedness \eqref{eq:intro-crookedness}:
\begin{equation}\label{eq:intro-milnor-formula}
    \TC[\gamma]=\frac12\int_{\Sph}\mu[\gamma,v]\,dA(v),
\end{equation}
where $dA$ is the usual area measure with $A(\Sph)=4\pi$.

Thus we find that both the total curvature and the bridge index are closely related to the concept of crookedness.
In particular, for our purpose, it is sufficient to prove the following pointwise estimate for crookedness.

\begin{theorem}[Pointwise crookedness estimate]\label{thm:milnor_pointwise}
    Let $K$ and $\gamma$ satisfy the assumptions in \Cref{thm:milnor}.
    Then
    \begin{equation}\label{eq:intro-directional-bound}
        \mu[\gamma,v]\ge \bri(K)
        \qquad\text{for almost every }v\in\Sph.
    \end{equation}
\end{theorem}

To prove this, we first show that for almost every $v$ the height function of a finite-total-curvature $C^1$ curve has only finitely many critical points, each a strict extremum; in addition, at the image of any critical point, all branches have the same tangent line and the same type of extremum (maximum or minimum). 
We then deform the corresponding branches of the approximating curve $\gamma_j$ simultaneously, preserving their knot type and leaving exactly one extremum on each branch. 
This produces a smooth representative $\beta_{j,v}$ of $K$ with $\mu[\beta_{j,v},v]=\mu[\gamma,v]$.
Using \eqref{eq:intro-bridge} gives \Cref{thm:milnor_pointwise}.
The detailed proof is given in \Cref{sec:milnor}.

\section{Milnor's inequality in the \texorpdfstring{$C^1$}{C1}-closure}\label{sec:milnor}

We begin with the proof of \Cref{thm:milnor_pointwise}, from which \Cref{thm:milnor} will follow immediately. 
We write $d_{\Sph}$ for the spherical distance.
A $C^1$ curve $\gamma:\T\to\R^3$ is called regular if its derivative never vanishes. Let $T:=\gamma'/|\gamma'|$ denote the unit tangent. The total curvature of $\gamma$ is then given by the spherical length of the tangent indicatrix $T:\T\to\Sph$ \cite[Proposition~3.1]{Sullivan}:
\begin{equation}\label{eq:tc-definition}
    \TC[\gamma]=\L_{\Sph}[T]
    :=\sup_{0=t_0<\cdots<t_m=1}
      \sum_{i=1}^m d_{\Sph}\bigl(T(t_{i-1}),T(t_i)\bigr).
\end{equation}
In particular, if $\gamma$ is parametrized by arclength and belongs to $W^{2,1}$, then $\TC[\gamma]=\int_\gamma|\gamma''|\,ds$.
More generally, finite total curvature is equivalent to bounded variation of the unit tangent.

\subsection{Good directions}

Here we justify the generic properties of height functions. 
For $v\in\Sph$, set
\begin{equation}\label{eq:height-critical}
    h_v(t):=\langle\gamma(t),v\rangle,
    \qquad Z_v:=\{t\in\T:h_v'(t)=0\}.
\end{equation}

We first show that, for a generic direction, the set of critical points is finite and consists of strict extrema.

\begin{proposition}\label{prop:generic-directions}
    Let $\gamma:\T\to\R^3$ be a regular $C^1$ closed curve with finite total curvature. There is a null set $E_0\subset\Sph$ such that, for every $v\notin E_0$, the set $Z_v$ is finite and $h_v'$ changes sign at every point of $Z_v$.
\end{proposition}

A subtle point is that $\gamma$ is merely $C^1$, so $h_v'$ need not be differentiable.
To overcome this, we work with the arclength reparametrization of the tangent indicatrix and establish transversality in that parameter.

\begin{lemma}\label{lem:tangent-factorization}
    Suppose that $\gamma:\T\to\R^3$ is a regular $C^1$ closed curve with $\ell:=\TC[\gamma]<\infty$. Then $\ell>0$, and there exist a continuous non-decreasing surjection $\sigma\colon[0,1]\to[0,\ell]$ and a Lipschitz curve $\tau\colon[0,\ell]\to\Sph$ such that
    \begin{equation}\label{eq:tangent-factorization}
        T=\tau\circ\sigma,
        \qquad |\tau'|=1\quad\text{a.e.}
    \end{equation}
    The set of $u$ for which $\sigma^{-1}(\{u\})$ is not a singleton is at most countable.
\end{lemma}

\begin{proof}
    If $\ell=0$, then $T$ is constant, contradicting $\gamma$ being closed. Hence $\ell>0$.
    
    Since $T$ is a rectifiable curve in $\Sph$ of length $\ell$,
    the standard arclength factorization
    (e.g., \cite[Proposition~2.5.9]{MR1835418})
    yields a continuous non-decreasing surjection
    $\sigma\colon[0,1]\to[0,\ell]$ and a Lipschitz map $\tau\colon[0,\ell]\to\Sph$ with $|\tau'|=1$ a.e.\ such that
    $T=\tau\circ\sigma$.

    Finally, the non-singleton fibers $\sigma^{-1}(\{u\})$ are pairwise disjoint intervals of positive length,
    and hence there are at most countably many such fibers.
\end{proof}

Now we prove the main generic properties for height functions.

\begin{proof}[Proof of \Cref{prop:generic-directions}]
    Let $T=\tau\circ\sigma$ be the factorization in \Cref{lem:tangent-factorization}.
    We first prove the key sign-changing property for $\tau$, and then translate it back to $T$.

    \emph{Step 1: Sign-changing zeros for the arclength reparametrization.}
    Choose an orthonormal basis $N_1(0),N_2(0)$ of $\tau(0)^\perp$ (i.e., the plane orthogonal to $\tau(0)$), and solve
    \begin{equation}\label{eq:normal-frame}
    N_i'=-\langle\tau',N_i\rangle\tau,
    \qquad i=1,2,
    \end{equation}
    on the interval $[0,\ell]$.
    This linear system has bounded measurable coefficients, so its solutions are Lipschitz. The equation yields $\langle\tau,N_i\rangle'=0$ and hence $\langle\tau,N_i\rangle=0$ on $[0,\ell]$.
    Then we also have $\langle N_i,N_j\rangle'=0$ for $i,j\in\{1,2\}$, so $N_1,N_2$ remain an orthonormal basis of $\tau^\perp$ at every $u\in[0,\ell]$.
    
    Consider the Lipschitz map
    \begin{equation}\label{eq:incidence-map}
        F(u,\theta)=\cos\theta\,N_1(u)+\sin\theta\,N_2(u),
        \qquad (u,\theta)\in[0,\ell]\times[0,2\pi].
    \end{equation}
    Since $F(u,\cdot)$ parametrizes the great circle perpendicular to $\tau(u)$ bijectively on $[0,2\pi)$, if we consider the zero set (except at the endpoints)
    \begin{equation}
        \widetilde Z_v:=\{u\in(0,\ell) : \langle\tau(u),v\rangle=0 \},
    \end{equation}
    then every $u\in\widetilde Z_v$ corresponds to a unique $\theta\in[0,2\pi)$ such that $F(u,\theta)=v$.
    
    We now remove the exceptional directions $v$ for the sign-changing property of $\langle \tau,v \rangle$.
    Let $A_0\subset[0,\ell]$ be a null set outside which $\tau,N_1,N_2$ are differentiable, \eqref{eq:normal-frame} holds, and $|\tau'|=1$.
    The image $F(A_0\times[0,2\pi])$ has area zero because $A_0\times[0,2\pi]$ has two-dimensional measure zero and $F$ is Lipschitz.
    We have $F_u=-\langle\tau',F\rangle\tau$ and $F_\theta=-\sin\theta\,N_1+\cos\theta\,N_2$ for a.e.\ $u$ and every $\theta$, so the Jacobian is given by
    \[
        JF:=|F_u\times F_\theta|=|\langle\tau',F\rangle|.
    \]
    On the set
    \begin{equation}
        X_0:=\{(u,\theta)\in(0,\ell)\times[0,2\pi):u\notin A_0,\
                       \langle\tau'(u),F(u,\theta)\rangle=0\},
    \end{equation}
    the Jacobian $JF$ vanishes, so applying the area formula \cite[Theorem~3.2.3]{Federer} to $F|_{X_0}$ shows that $F(X_0)$ also has area zero.
    Define the null set
    \[
    E':=F(A_0\times[0,2\pi])\cup F(X_0).
    \]
    For $v\in\Sph\setminus E'$ and $u_0\in\widetilde Z_v$, let $\theta_0\in[0,2\pi)$ be the unique number such that $F(u_0,\theta_0)=v$. 
    Then $u_0\notin A_0$ and $(u_0,\theta_0)\notin X_0$, so $\langle\tau'(u_0),v\rangle\ne0$.
    Therefore, around each zero $u_0\in \widetilde Z_v$, we have
    \[
        \langle\tau(u),v\rangle
        =
        \langle\tau'(u_0),v\rangle(u-u_0)
        +o(|u-u_0|) \qquad \text{with $\langle\tau'(u_0),v\rangle\ne0$},
    \]
    so $\langle\tau(\cdot),v\rangle$ has opposite signs on the two sides of $u_0$. In particular, every point of $\widetilde Z_v$ is isolated.
    (Note carefully that, at this stage, $\widetilde Z_v$ may still be infinite since accumulation at the endpoints has not yet been excluded.)

    \emph{Step 2: Translate back to the original parameter.}
    Let $\{u_j\}_{j\geq1}\subset[0,\ell]$ be the at most countable family of points such that $\sigma^{-1}(u_j)$ is not a singleton, and let $C_j\subset\Sph$ be the great circle perpendicular to $\tau(u_j)$.
    Also let $C_0\subset\Sph$ be the great circle perpendicular to $T(0)$.
    Now we define the null set
    \[
        E_0 := E' \cup \bigcup_{j\geq 0}C_j.
    \]
    
    Fix any $v\in\Sph\setminus E_0$. 
    Since $\tau(0)=\tau(\ell)=T(0)$ and $v\notin C_0$, neither $0$ nor $\ell$ is a zero of $\langle\tau,v\rangle$. Hence $\widetilde Z_v$ agrees with the full zero set of $\langle\tau,v\rangle$ on $[0,\ell]$.
    This zero set is compact, and each element is isolated by Step~1. Therefore $\widetilde Z_v$ is finite.
    In addition, for every $u_0\in\widetilde Z_v$ there is a unique $t_0\in(0,1)$ such that $\sigma(t_0)=u_0$, since otherwise $u_0=u_j$ for some $j\geq1$ and hence $v\in C_j$.
    Since
    \[
        h_v'(t)=|\gamma'(t)|\langle\tau(\sigma(t)),v\rangle
    \]
    and $|\gamma'|>0$, the map $\sigma$ induces a bijection between $Z_v$ and $\widetilde Z_v$. In particular, $\#Z_v=\#\widetilde Z_v<\infty$.
    Moreover, monotonicity and continuity of $\sigma$ give $\sigma(t)<u_0$ for $t<t_0$ and $\sigma(t)>u_0$ for $t>t_0$.
    Hence the sign-changing property of $\langle \tau(u),v \rangle$ at $u_0$ transfers to $\langle T(t),v\rangle$ at $t_0$, and then to $h_v'(t)$.
\end{proof}

We next show that, by additionally removing suitable directions, all branches at the image of a given critical point are of the ``same type'' and thus can be treated by the same local deformation.
For a regular $C^1$ closed curve $\gamma$, we define the set
\begin{equation}\label{eq:double-pairs}
    D:=\{(s,t)\in\T^2:s\ne t,\ \gamma(s)=\gamma(t)\}.
\end{equation}
Let $D_{\mathrm{iso}}$ be the set of isolated points of $D$, and let $A\subset\T$ consist of all coordinates of pairs in $D_{\mathrm{iso}}$.

\begin{lemma}\label{lem:fibers}
    Let $\gamma$ be a regular $C^1$ closed curve.
    For every $p\in\R^3$, the fiber $\gamma^{-1}(\{p\})$ is finite. Moreover, $D_{\mathrm{iso}}$ and $A$ are at most countable, and the set
    \begin{equation}\label{eq:isolated-exceptional}
        E_1:=\bigcup_{s\in A}\{v\in\Sph:\langle T(s),v\rangle=0\}
    \end{equation}
    has area zero.
\end{lemma}

\begin{proof}
    Since $\gamma$ is regular and $C^1$, it is locally injective: near each parameter, projection onto the tangent direction is strictly monotone. 
    Thus each fiber $\gamma^{-1}(\{p\})$ is discrete and compact, and hence finite.
    The space $D$ is second countable. Assigning to each isolated point a basis element containing no other point of $D$ shows that $D_{\mathrm{iso}}$ is at most countable. Hence the same holds for $A$. Finally, $E_1$ is thus a countable union of great circles, so has measure zero.
\end{proof}

\begin{proposition}\label{prop:compatible}
    Suppose that $\TC[\gamma]<\infty$, and let $v\notin E_0\cup E_1$.
    If $s\in Z_v$, and if $t\in\T$ satisfies $\gamma(t)=\gamma(s)$, then $\R\gamma'(t)=\R\gamma'(s)$; in particular, $t\in Z_v$.
    Moreover, the points $s,t$ are either both strict local maxima or both strict local minima of $h_v$.
\end{proposition}

\begin{proof}
    We may assume $s\ne t$.
    We first show that $t\in Z_v$.
    Since $s\in Z_v$, it suffices to show that $\R\gamma'(t)=\R\gamma'(s)$.
    Suppose on the contrary that $\gamma'(s)$ and $\gamma'(t)$ are linearly independent. Let $P$ be the plane spanned by them, and $\pi_P:\R^3\to P$ be the orthogonal projection. 
    Then the map $(u,w)\mapsto\pi_P(\gamma(u)-\gamma(w))$ from $\T^2$ to $P\simeq\R^2$ has invertible derivative at $(s,t)\in\T^2$, since its columns are $\gamma'(s)$ and $-\gamma'(t)$. The inverse function theorem, together with $\gamma(s)=\gamma(t)$, implies that $(s,t)$ is the only zero in a neighborhood. In particular, $(s,t)\in D_{\mathrm{iso}}$. Then $s\in A$.
    On the other hand, we have $\langle T(s),v\rangle=0$ by $s\in Z_v$, contradicting $v\notin E_1$. 
    This proves $\R\gamma'(t)=\R\gamma'(s)$, and in particular $t\in Z_v$.

    By \Cref{prop:generic-directions}, both the critical points $s,t$ are strict extrema of $h_v$. If $s$ were a local maximum and $t$ a local minimum, then there would exist disjoint neighborhoods $U_s$ and $U_t$ of $s$ and $t$, respectively, such that the height $c:=h_v(s)=h_v(t)$ satisfies $h_v(u)<c$ for all $u\in U_s\setminus\{s\}$ while $h_v(w)>c$ for all $w\in U_t\setminus\{t\}$.
    The equality $\gamma(u)=\gamma(w)$ could then hold only at $(u,w)=(s,t)$ in $U_s\times U_t$, again ensuring $(s,t)\in D_{\mathrm{iso}}$. This contradicts $v\notin E_1$. The remaining case can be treated identically, so the proof is complete.
\end{proof}

We now introduce the notion of a good direction, which will be used throughout the rest of this section.

\begin{definition}
    We call $v\in \Sph\setminus (E_0\cup E_1)$ a \emph{good direction}, where $E_0$ (resp.\ $E_1$) is as in \Cref{prop:generic-directions} (resp.\ \Cref{lem:fibers}).
\end{definition}

For every good direction, if $s\in\T$ is a critical point of $h_v$, then any preimage of $\gamma(s)$ is also critical, with the same tangent line and the same type of extremum.

\subsection{Simultaneous deformation of the branches}

Now we prove a key deformation result for the critical branches in every good direction.

We first record the stability and localization statements needed to deform the branches.
The following $C^1$-stability result is already discussed by Blatt \cite{Blatt}.

\begin{lemma}\label{lem:embedding-stability}
    Every regular $C^1$ embedding $\eta\colon\T\to\R^3$ has a $C^1$ neighborhood consisting of regular embeddings ambient isotopic to $\eta$. Consequently, under the assumptions of \Cref{thm:milnor}, the approximating curves may be assumed $C^\infty$ smooth.
\end{lemma}

\begin{proof}
    The first assertion follows directly from \cite[Corollary~1.5]{Blatt}, and then the second follows by a standard mollifier argument.
\end{proof}

The following elementary local graph representation is a standard consequence of the inverse function theorem, compactness, and $C^1$ convergence, so the detailed proof is safely omitted. We just record the precise statement needed below.
Hereafter we write $B_r^d(p)\subset\R^d$ for the open ball of radius $r$ centered at $p$.

\begin{lemma}\label{lem:cylinders}
    Let $\gamma:\T\to\R^3$ be a regular $C^1$ closed curve and $p\in\gamma(\T)$.
    Suppose that there is a one-dimensional linear subspace $L\subset\R^3$ such that $\R\gamma'(s)=L$ for all $s\in\gamma^{-1}(p)$.
    Then the following properties hold:
    \begin{itemize}
        \item \textup{(Local graph representation)} Choose orthogonal coordinates $(x,w)\in\R\times\R^2$ centered at $p$, with the $x$-axis along $L$. For all sufficiently small $r>0$, the preimage $\gamma^{-1}(C_r)$ of the closed cylinder
        \begin{equation}\label{eq:cylinder}
            C_r:=[-r,r]\times \overline{B_r^2(0)},
        \end{equation}
        consists of pairwise disjoint closed intervals $I_1,\dots,I_m\subset \T$, where $m:=\#\gamma^{-1}(\{p\})$.
        For every $i$, the branch $\gamma|_{I_i}$ can be written as the graph of a function $f_i\in C^1([-r,r];\R^2)$, namely
        \begin{equation}
            \gamma(I_i)=\{ (x,f_i(x)) : x\in[-r,r] \},
        \end{equation}
        such that
        \begin{equation}\label{eq:limiting-graphs}
            f_i(0)=f_i'(0)=0,
            \qquad \max_{1\leq i\leq m}\max_{|x|\le r}|f_i(x)|<r/8.
        \end{equation}
        \item \textup{(Graphical convergence)} If $\eta_j\to\gamma$ in $C^1$, then, for all sufficiently large $j$, the preimage $\eta_j^{-1}(C_r)$ also consists of exactly $m$ pairwise disjoint compact parameter intervals. On the corresponding intervals, the branches of $\eta_j$ can be written as graphs of $f_{j,i}\in C^1([-r,r];\R^2)$ such that
        \begin{equation}\label{eq:graph-convergence}
            \max_{1\leq i\leq m}\max_{|x|\le r}|f_{j,i}(x)|<r/4,
        \end{equation}
        and, after relabeling $i$ if necessary, $f_{j,i}\to f_i$ in $C^1$ as $j\to\infty$.
    \end{itemize}
    In particular, for finitely many distinct points $p$, the above cylinders $C_r$ may be chosen pairwise disjoint.
\end{lemma}




Now we turn to the construction of deformations.
The following local construction preserves pairwise disjointness of the branches while replacing their height functions by strictly concave ones near the center.

\begin{lemma}\label{lem:simultaneous-deformation}
    Let $\eta$ be a smooth embedded closed curve such that the preimage of the cylinder $C_r$
    consists of pairwise disjoint closed intervals $I_1,\dots,I_m\subset\T$, and the corresponding branches $\eta|_{I_i}$ are given by smooth graphs
    \begin{equation}
        \eta(I_i)
        =\{(x,f_i(x)):x\in[-r,r]\},
        \quad  f_i=(y_i,z_i),
        \quad \max_{1\leq i\leq m}\max_{|x|\le r}|f_i(x)|<r/4.
    \end{equation}
    Assume that for every $i$,
    \begin{align}\label{eq:z-signs}
        \begin{split}
        &z_i(x)<0 \quad (r/3\le|x|\le r),\\
        &z_i'(x)>0 \quad (-r\le x\le-r/3), \qquad z_i'(x)<0 \quad (r/3\le x\le r).
        \end{split}
    \end{align}
    Then there exists a smooth ambient isotopy $\Phi:\R^3\times[0,1]\to\R^3$, with $\Phi_\lambda:=\Phi(\cdot,\lambda)$, and a compact set $S$ in the interior of $C_r$ such that $\Phi_0=\mathrm{Id}$, and $\Phi_\lambda(p)=p$ for every $p\in\R^3\setminus S$ and $\lambda\in[0,1]$.
    Moreover, for every $i$,
    \begin{equation}
        \Phi_1(\eta(I_i))
        =\{(x,\widehat f_i(x)):x\in[-r,r]\},
        \qquad \widehat f_i=(\widehat y_i,\widehat z_i),
    \end{equation}
    where there is a unique $x_i\in(-r/3,r/3)$ satisfying
    \begin{equation}\label{eq:deformed-height-signs}
        \widehat z_i'(x)>0\quad(-r\le x<x_i),\qquad
        \widehat z_i'(x_i)=0,\qquad
        \widehat z_i'(x)<0\quad(x_i<x\le r).
    \end{equation}
    
    Similarly, if all signs in \eqref{eq:z-signs} are reversed, the same conclusion holds with all signs in \eqref{eq:deformed-height-signs} reversed.
\end{lemma}

\begin{proof}
    Choose $\varepsilon>0$ such that $\varepsilon r^2<r/4$, and set
    \begin{equation}
        q(x):=\varepsilon(r^2-x^2),\qquad Q(x):=(0,q(x)).
    \end{equation}
    Let $M:=\max_i\|z_i''\|_{L^\infty(-r/3,r/3)}$ and choose $\delta$ such that $0<\delta<\min\{1,\frac{\varepsilon}{M+2\varepsilon}\}$.
    Take a smooth even function $\alpha\colon\R\to[\delta,1]$ such that
    \begin{equation}\label{eq:contraction-cutoff}
        \alpha=\delta\ \text{on }[-r/3,r/3],\qquad
        \alpha=1\ \text{for }|x|\ge2r/3,\qquad
        \alpha'(x)\ge0\ \text{for }x>0.
    \end{equation}
    For $0\le\lambda\le1$, define
    \begin{equation}\label{eq:branch-isotopy}
        \alpha_\lambda:=1-\lambda+\lambda\alpha,
        \qquad f_i^\lambda:=Q+\alpha_\lambda(f_i-Q).
    \end{equation}
    These graphs satisfy $f_i^\lambda(x) \in B_{r/4}^2(0)$ for all $x\in[-r,r]$ since they are convex combinations of $Q$ and $f_i$ such that $Q(x),f_i(x)\in B_{r/4}^2(0)$ for $x\in[-r,r]$. Their differences satisfy
    \begin{equation}\label{eq:branch-differences}
        f_i^\lambda(x)-f_j^\lambda(x)
        =\alpha_\lambda(x)(f_i(x)-f_j(x))\ne0
        \qquad(i\ne j),
    \end{equation}
    because $\alpha_\lambda\ge\delta>0$ and $\eta$ is embedded.

    We now realize \eqref{eq:branch-isotopy} by an ambient isotopy. Choose a cutoff $\chi\in C_c^\infty(B_{r/2}^2;[0,1])$ with $\chi=1$ on $\overline{B_{r/3}^2}$, and consider the smooth time-dependent vector field
    \begin{equation}\label{eq:ambient-vector-field}
        V_\lambda(x,w)
        =\left(0,\frac{\alpha(x)-1}{\alpha_\lambda(x)}
                 \chi(w)(w-Q(x))\right).
    \end{equation}
    The support of $V$ is contained in the interior of $C_r$. The associated flow $\Phi_\lambda$ exists for $0\le\lambda\le1$, consists of smooth diffeomorphisms, and is the identity outside the interior of  $C_r$. Along the graphs $f_i^\lambda$, the cutoff equals one and
    \begin{equation}
        \partial_\lambda f_i^\lambda
        =(\alpha-1)(f_i-Q)
        =\frac{\alpha-1}{\alpha_\lambda}(f_i^\lambda-Q).
    \end{equation}
    Uniqueness for this ODE gives
    \[
        \Phi_\lambda(x,f_i(x))=(x,f_i^\lambda(x)).
    \]
    Thus, setting $\widehat f_i:=f_i^1$, the prescribed deformation is realized by the ambient isotopy.

    It remains to verify \eqref{eq:deformed-height-signs}. We have
    $\widehat z_i=q+\alpha(z_i-q)$ and hence
    \begin{equation}\label{eq:deformed-height-derivative}
        \widehat z_i'=(1-\alpha)q'+\alpha z_i'+\alpha'(z_i-q).
    \end{equation}
    On the interval $[r/3,r]$ we have $q'<0$, $z_i'<0$, $z_i-q\leq z_i<0$, and $\alpha'\ge0$, so $\widehat z_i'<0$.
    Similarly, 
    $\widehat z_i'>0$ on $[-r,-r/3]$. 
    On $[-r/3,r/3]$, since $\alpha=\delta$,
    \begin{equation}
        \widehat z_i''=-2\varepsilon(1-\delta)+\delta z_i''
        \le-2\varepsilon+\delta(M+2\varepsilon)<-\varepsilon.
    \end{equation}
    Thus $\widehat z_i'$ is strictly decreasing on $[-r/3,r/3]$, positive at $-r/3$, and negative at $r/3$. Hence $\widehat z_i'$ has a unique zero $x_i\in(-r/3,r/3)$, and \eqref{eq:deformed-height-signs} holds. If all signs are reversed, apply the same argument after reversing the $z$-coordinate.
\end{proof}


We now apply the local deformation at all critical image points of a fixed good direction.

\begin{proposition}\label{prop:directional-recovery}
    Suppose that smooth regular embeddings $\gamma_j:\T\to\R^3$ converge in $C^1$ as $j\to\infty$ to a regular curve $\gamma:\T\to\R^3$ with finite total curvature. 
    For each good direction $v\in\Sph$ and all sufficiently large $j$, there is a smooth embedding $\beta_{j,v}:\T\to\R^3$ ambient isotopic to $\gamma_j$ such that
    \begin{equation}\label{eq:exact-recovery}
        \mu[\beta_{j,v},v]=\mu[\gamma,v].
    \end{equation}
\end{proposition}

\begin{proof}
    Fix a good direction $v\in\Sph$, and let $p_1,\ldots,p_N\in\R^3$ be the distinct images of the finitely many critical points of $h_v$. For each $\ell$, write
    \[
        \gamma^{-1}(\{p_\ell\})=\{t_{\ell,1},\ldots,t_{\ell,m_\ell}\}.
    \]
    By \Cref{prop:compatible}, all $t_{\ell,i}$ are critical points of the same type and their tangent lines agree; denote this common line by $L_\ell$. For each $\ell$ we use orthogonal coordinates $(x,y,z)$ centered at $p_\ell$, with the $x$-axis along $L_\ell$ and the $z$-axis in direction $v$. By \Cref{lem:cylinders}, we may choose $r_\ell>0$ so that the corresponding closed cylinders $C_{r_\ell}$ are pairwise disjoint and
    \[
    \gamma^{-1}(C_{r_\ell})=I_{\ell,1}\cup\cdots\cup I_{\ell,m_\ell},
    \]
    where the $I_{\ell,i}$ are pairwise disjoint closed intervals and the corresponding branches $\gamma|_{I_{\ell,i}}$ are written as graphs $f_{\ell,i}=(y_{\ell,i},z_{\ell,i})$ on $[-r_\ell,r_\ell]$.

    Suppose first that the points $t_{\ell,i}$ are local maxima of $h_v$. By the sign-changing property of $h_v'$ at $t_{\ell,i}$, taking smaller $r_\ell$ if necessary, we have
    \begin{equation}\label{eq:limit-branch-signs}
        z_{\ell,i}(0)=0,\quad
        z_{\ell,i}'(x)>0\quad(-r_\ell\le x<0),\quad
        z_{\ell,i}'(x)<0\quad(0<x\le r_\ell).
    \end{equation}
    In particular, $z_{\ell,i}(x)<0$ for $x\ne0$.
    If the $t_{\ell,i}$ are local minima, all these signs are reversed.

    By the graphical convergence in \Cref{lem:cylinders}, for all sufficiently large $j$,
    \[
        \gamma_j^{-1}(C_{r_\ell})=I_{j,\ell,1}\cup\cdots\cup I_{j,\ell,m_\ell},
    \]
    and the smooth graphs $f_{j,\ell,i}=(y_{j,\ell,i},z_{j,\ell,i})$ corresponding to the branches $\gamma_j|_{I_{j,\ell,i}}$ converge to $f_{\ell,i}$ in $C^1([-r_\ell,r_\ell])$ as $j\to\infty$. Since there are only finitely many $\ell$ and $i$, we may take $j$ large enough that this holds simultaneously for all branches. In the case that the critical points are local maxima, property \eqref{eq:limit-branch-signs} and compactness of $\{r_\ell/3\le|x|\le r_\ell\}$ imply
    \begin{align}
        &z_{j,\ell,i}(x)<0\ (r_\ell/3\le|x|\le r_\ell),\\
        &z_{j,\ell,i}'(x)>0\ (-r_\ell\le x\le-r_\ell/3),\qquad z_{j,\ell,i}'(x)<0\ (r_\ell/3\le x\le r_\ell),
    \end{align}
    while in the local-minimum case all signs are reversed. Thus  the hypotheses of \Cref{lem:simultaneous-deformation} hold in every cylinder $C_{r_\ell}$.

    It remains to exclude critical points of the approximating heights outside these cylinders. For every $\ell,i$, let $V_{\ell,i}\subset I_{\ell,i}$ be the subinterval corresponding to $|x|<r_\ell/4$. The union of the intervals $V_{\ell,i}$ contains $Z_v$. Hence $|h_v'|$ has a positive minimum on the compact set
    \[
        \T\setminus\bigcup_{\ell,i}V_{\ell,i}.
    \]
    Since $\gamma_j\to\gamma$ in $C^1$, we have $\langle\gamma_j',v\rangle\to h_v'$ uniformly, so for all sufficiently large $j$ the height of $\gamma_j$ has no critical point on this compact set. On the other hand, $C^1$ convergence and the choice $|x|<r_\ell/4$ imply that $\gamma_j(V_{\ell,i})\subset C_{r_\ell}$ for all sufficiently large $j$. Therefore every critical point of the height of $\gamma_j$ lies in one of $C_{r_\ell}$.

    Fix such a large $j$. In each cylinder $C_{r_\ell}$, we apply \Cref{lem:simultaneous-deformation} at $p_\ell$ and write $\Phi^{\ell}$ for the resulting ambient isotopy. Since the cylinders are pairwise disjoint and each isotopy is supported in the interior of its cylinder, if we set
    \[
        \beta_{j,v}:=\Phi^{1}_1\circ\cdots\circ\Phi^{N}_1\circ\gamma_j,
    \]
    then $\beta_{j,v}$ is a smooth embedding ambient isotopic to $\gamma_j$. Outside the cylinders its height has no critical points, while inside each $C_{r_\ell}$ every branch has exactly one critical point, of the same type as $t_{\ell,i}$. Finally, if we write $I^+$ for the set of all indices $1\le\ell\le N$ such that $\gamma^{-1}(\{p_\ell\})$ consists of local maxima, then
    \[
        \mu[\beta_{j,v},v]
        =\sum_{\ell \in I^+}m_\ell
        =\mu[\gamma,v],
    \]
    which proves \eqref{eq:exact-recovery}.
\end{proof}

\subsection{Completion of the proof of Milnor's inequality}

We now complete the proof of the pointwise crookedness estimate and the extension of Milnor's inequality.

\begin{proof}[Proof of \Cref{thm:milnor_pointwise}]
    By \Cref{lem:embedding-stability}, the approximating embeddings $\gamma_j$ may be taken smooth. 
    For each good direction $v$, \Cref{prop:directional-recovery} and \eqref{eq:intro-bridge} give
    \begin{equation}
        \bri(K)\le\mu[\beta_{j,v},v]=\mu[\gamma,v].
    \end{equation}
    Good directions have full area measure, proving \eqref{eq:intro-directional-bound}. 
\end{proof}

\begin{proof}[Proof of \Cref{thm:milnor}]
    The conclusion follows immediately from \Cref{thm:milnor_pointwise} and Milnor's integral formula \eqref{eq:intro-milnor-formula}.
\end{proof}


We also record an immediate consequence for the bending energy.

\begin{corollary}\label{cor:bending-bound}
    Under the hypotheses of \Cref{thm:milnor}, suppose that the arclength parametrization of $\gamma$ belongs to $H^2$. Then
    \begin{equation}\label{eq:normalized-bending-bound}
        \L[\gamma]\B[\gamma]\ge(2\pi \bri(K))^2.
    \end{equation}
    In the equality case, $|\kappa|=2\pi \bri(K)/\L[\gamma]$ holds almost everywhere.
    In particular, the arclength parametrization is of class $W^
    {2,\infty}$, or equivalently $C^{1,1}$.
\end{corollary}

\begin{proof}
    By \Cref{thm:milnor} and the Cauchy--Schwarz inequality,
    \begin{equation}
        (2\pi \bri(K))^2
        \le\left(\int_\gamma|\kappa|\,ds\right)^2
        \le \L[\gamma]\int_\gamma|\kappa|^2\,ds.
    \end{equation}
    Equality forces equality in both inequalities: the equality in Cauchy--Schwarz makes $|\kappa|$ constant a.e., and the remaining condition yields the desired value.
    The regularity follows since $|\gamma_{ss}|$ is constant a.e.
\end{proof}

Finally, we complete the proof of the circular elastic knot conjecture.

\begin{proof}[Proof of \Cref{thm:main}]
    By \Cref{thm:milnor}, every unit-speed $H^2$ curve in the $C^1$-closure of $\CK$ satisfies
    \[
        \TC[\gamma]\ge 2\pi\bri(K).
    \]
    Thus Reiter--von der Mosel's hypothesis \cite[Assumption~1.1]{ReiterVdM2026} is satisfied. Since $\bri(K)=\bra(K)=a$, their result \cite[Theorem~1.2]{ReiterVdM2026} applies and shows that every elastic knot for $K$ is, up to a Euclidean isometry and reparametrization, the $a$-fold covered circle. Since all curves in our definition are unit-speed, the remaining reparametrization of the circle can be absorbed into a Euclidean isometry. This proves the claim.
\end{proof}

\begin{acknowledgements}
    This work is supported by JSPS KAKENHI Grant Numbers JP23H00085 and JP24K00532.
    The author would like to thank Yasuhiko Asao and Olivier Pierre-Louis for helpful discussions during the very early stages of this work.
\end{acknowledgements}

\begin{AI}
    ChatGPT was used interactively in the preparation and revision of this manuscript. The author independently reviewed and revised all mathematical statements, proofs, and exposition and takes full responsibility for the final manuscript.
\end{AI}

\bibliography{references}

@techreport{Blatt,
  author      = {Blatt, Simon},
  title       = {Note on continuously differentiable isotopies},
  type        = {Report},
  number      = {34},
  institution = {Institut f\"ur Mathematik, RWTH Aachen University},
  year        = {2009},
  url         = {https://www.instmath.rwth-aachen.de/Preprints/blatt20090825.pdf},
  note        = {Available at \url{https://www.instmath.rwth-aachen.de/Preprints/blatt20090825.pdf}},
}

@book{Federer,
  author    = {Federer, Herbert},
  title     = {Geometric measure theory},
  series    = {Classics in Mathematics},
  publisher = {Springer-Verlag},
  address   = {Berlin},
  year      = {1996},
  doi       = {10.1007/978-3-642-62010-2},
  note      = {Reprint of the 1969 edition. DOI \href{https://doi.org/10.1007/978-3-642-62010-2}{\nolinkurl{10.1007/978-3-642-62010-2}}},
}

@article{GRM,
  author  = {Gerlach, Henryk and Reiter, Philipp and von der Mosel, Heiko},
  title   = {The elastic trefoil is the doubly covered circle},
  journal = {Arch. Ration. Mech. Anal.},
  volume  = {225},
  year    = {2017},
  number  = {1},
  pages   = {89--139},
  doi     = {10.1007/s00205-017-1100-9},
  note    = {DOI \href{https://doi.org/10.1007/s00205-017-1100-9}{\nolinkurl{10.1007/s00205-017-1100-9}}},
}

@article{Milnor,
  author  = {Milnor, John W.},
  title   = {On the total curvature of knots},
  journal = {Ann. of Math. (2)},
  volume  = {52},
  year    = {1950},
  number  = {2},
  pages   = {248--257},
  doi     = {10.2307/1969467},
  note    = {DOI \href{https://doi.org/10.2307/1969467}{\nolinkurl{10.2307/1969467}}},
}

@incollection{Sullivan,
  author    = {Sullivan, John M.},
  title     = {Curves of finite total curvature},
  booktitle = {Discrete differential geometry},
  series    = {Oberwolfach Seminars},
  volume    = {38},
  publisher = {Birkh\"auser},
  address   = {Basel},
  year      = {2008},
  pages     = {137--161},
  doi       = {10.1007/978-3-7643-8621-4_7},
  note      = {DOI \href{https://doi.org/10.1007/978-3-7643-8621-4_7}{\nolinkurl{10.1007/978-3-7643-8621-4_7}}},
}

@phdthesis{Wacker,
  author      = {Wacker, Elisabeth},
  title       = {Total curvature of curves in the {$C^1$}-closure of knot classes with finite, isolated self intersections},
  type        = {Ph.D. thesis},
  school      = {RWTH Aachen University},
  institution = {RWTH Aachen University},
  year        = {2021},
  url         = {https://publications.rwth-aachen.de/record/836835/files/836835.pdf},
  note        = {Available at \url{https://publications.rwth-aachen.de/record/836835/files/836835.pdf}},
}

@article{GPL,
  author  = {Gallotti, Riccardo and Pierre-Louis, Olivier},
  title   = {Stiff knots},
  journal = {Phys. Rev. E (3)},
  volume  = {75},
  year    = {2007},
  number  = {3},
  pages   = {Paper No. 031801},
  doi     = {10.1103/PhysRevE.75.031801},
  note    = {DOI \href{https://doi.org/10.1103/PhysRevE.75.031801}{\nolinkurl{10.1103/PhysRevE.75.031801}}},
}

@article{vdM,
  author  = {von der Mosel, Heiko},
  title   = {Minimizing the elastic energy of knots},
  journal = {Asymptot. Anal.},
  volume  = {18},
  year    = {1998},
  number  = {1--2},
  pages   = {49--65},
  doi     = {10.3233/ASY-1998-314},
  note    = {DOI \href{https://doi.org/10.3233/ASY-1998-314}{\nolinkurl{10.3233/ASY-1998-314}}},
}

@book{MR1835418,
  author    = {Burago, Dmitri and Burago, Yuri and Ivanov, Sergei},
  title     = {A course in metric geometry},
  series    = {Graduate Studies in Mathematics},
  volume    = {33},
  publisher = {American Mathematical Society},
  address   = {Providence, RI},
  year      = {2001},
  pages     = {xiv+415},
  doi       = {10.1090/gsm/033},
  note      = {DOI \href{https://doi.org/10.1090/gsm/033}{\nolinkurl{10.1090/gsm/033}}},
}

@article{GM99,
  author  = {Gonzalez, Oscar and Maddocks, John H.},
  title   = {Global curvature, thickness, and the ideal shapes of knots},
  journal = {Proc. Natl. Acad. Sci. USA},
  volume  = {96},
  year    = {1999},
  number  = {9},
  pages   = {4769--4773},
  doi     = {10.1073/pnas.96.9.4769},
  note    = {DOI \href{https://doi.org/10.1073/pnas.96.9.4769}{\nolinkurl{10.1073/pnas.96.9.4769}}},
}

@article{LS85,
  author  = {Langer, Joel and Singer, David A.},
  title   = {Curve straightening and a minimax argument for closed elastic curves},
  journal = {Topology},
  volume  = {24},
  year    = {1985},
  number  = {1},
  pages   = {75--88},
  doi     = {10.1016/0040-9383(85)90046-1},
  note    = {DOI \href{https://doi.org/10.1016/0040-9383(85)90046-1}{\nolinkurl{10.1016/0040-9383(85)90046-1}}},
}

@book{BZH,
  author    = {Burde, Gerhard and Zieschang, Heiner and Heusener, Michael},
  title     = {Knots},
  edition   = {3},
  series    = {De Gruyter Studies in Mathematics},
  volume    = {5},
  publisher = {De Gruyter},
  address   = {Berlin},
  year      = {2013},
  doi       = {10.1515/9783110270785},
  note      = {DOI \href{https://doi.org/10.1515/9783110270785}{\nolinkurl{10.1515/9783110270785}}},
}

@misc{ReiterVdM2026,
  author        = {Reiter, Philipp and von der Mosel, Heiko},
  title         = {Elastic {BB} knots are multifold circles},
  year          = {2026},
  eprint        = {2609.22575v1},
  archiveprefix = {arXiv},
  primaryclass  = {math.DG},
  url           = {https://arxiv.org/abs/2609.22575v1},
  note          = {Preprint, \href{https://arxiv.org/abs/2609.22575v1}{arXiv:2609.22575v1}},
}

@article{DER,
  author  = {Diao, Yuanan and Ernst, Claus and Reiter, Philipp},
  title   = {Knots with equal bridge index and braid index},
  journal = {J. Knot Theory Ramifications},
  volume  = {30},
  year    = {2021},
  number  = {11},
  pages   = {Paper No. 2150075},
  doi     = {10.1142/S0218216521500759},
  note    = {DOI \href{https://doi.org/10.1142/S0218216521500759}{\nolinkurl{10.1142/S0218216521500759}}},
}

@article{GRvM,
  author  = {Gilsbach, Alexandra and Reiter, Philipp and von der Mosel, Heiko},
  title   = {Symmetric elastic knots},
  journal = {Math. Ann.},
  volume  = {385},
  year    = {2023},
  number  = {1--2},
  pages   = {811--844},
  doi     = {10.1007/s00208-021-02346-9},
  note    = {DOI \href{https://doi.org/10.1007/s00208-021-02346-9}{\nolinkurl{10.1007/s00208-021-02346-9}}},
}

@article{MR4861585,
  author   = {Miura, Tatsuya and M\"uller, Marius and Rupp, Fabian},
  title    = {Optimal thresholds for preserving embeddedness of elastic flows},
  journal  = {Amer. J. Math.},
  volume   = {147},
  year     = {2025},
  number   = {1},
  pages    = {33--80},
  doi      = {10.1353/ajm.2025.a950273},
  note     = {DOI \href{https://doi.org/10.1353/ajm.2025.a950273}{\nolinkurl{10.1353/ajm.2025.a950273}}},
}

@article{BR,
  author  = {Bartels, S{\"o}ren and Reiter, Philipp},
  title   = {Stability of a simple scheme for the approximation of elastic knots and self-avoiding inextensible curves},
  journal = {Math. Comp.},
  volume  = {90},
  year    = {2021},
  number  = {330},
  pages   = {1499--1526},
  doi     = {10.1090/mcom/3633},
  note    = {DOI \href{https://doi.org/10.1090/mcom/3633}{\nolinkurl{10.1090/mcom/3633}}},
}

@incollection{MR5054189,
  author    = {Miura, Tatsuya},
  title     = {Elastic curves and self-intersections},
  booktitle = {2024 {MATRIX} annals. {P}art {I}},
  series    = {{MATRIX} Book Series},
  volume    = {7},
  publisher = {Springer},
  address   = {Cham},
  year      = {[2026] \copyright 2026},
  pages     = {337--374},
  doi       = {10.1007/978-3-032-16202-1_14},
  note      = {DOI \href{https://doi.org/10.1007/978-3-032-16202-1_14}{\nolinkurl{10.1007/978-3-032-16202-1_14}}},
}

@book{MR4439733,
  editor    = {Adams, Colin and Flapan, Erica and Henrich, Allison and Kauffman, Louis H. and Ludwig, Lewis D. and Nelson, Sam},
  title     = {Encyclopedia of knot theory},
  publisher = {CRC Press},
  address   = {Boca Raton, FL},
  year      = {[2021] \copyright 2021},
  pages     = {xi+941},
  doi       = {10.1201/9781138298217},
  note      = {DOI \href{https://doi.org/10.1201/9781138298217}{\nolinkurl{10.1201/9781138298217}}},
}

@article{MR4631455,
  author   = {Miura, Tatsuya},
  title    = {Li-{Y}au type inequality for curves in any codimension},
  journal  = {Calc. Var. Partial Differential Equations},
  volume   = {62},
  year     = {2023},
  number   = {8},
  pages    = {Paper No. 216, 28 pp.},
  doi      = {10.1007/s00526-023-02559-7},
  note     = {DOI \href{https://doi.org/10.1007/s00526-023-02559-7}{\nolinkurl{10.1007/s00526-023-02559-7}}},
}

@article{KrishnaMorton,
  author  = {Krishna, Siddhi and Morton, Hugh},
  title   = {Twist positivity, {$L$}-space knots, and concordance},
  journal = {Selecta Math. (N.S.)},
  volume  = {31},
  year    = {2025},
  number  = {1},
  pages   = {Paper No. 11, 27 pp.},
  doi     = {10.1007/s00029-024-00996-6},
  note    = {DOI \href{https://doi.org/10.1007/s00029-024-00996-6}{\nolinkurl{10.1007/s00029-024-00996-6}}},
}

@article{Himeno,
  author  = {Himeno, Keisuke},
  title   = {The bridge index and the braid index for twist positive knots},
  journal = {Proc. Amer. Math. Soc.},
  year    = {2026},
  doi     = {10.1090/proc/17938},
  note    = {To appear. DOI \href{https://doi.org/10.1090/proc/17938}{\nolinkurl{10.1090/proc/17938}}},
}

\end{document}